\documentclass{article}
\usepackage[utf8]{inputenc}
\usepackage{amsmath,amsfonts,amsthm,amssymb,amsthm}
\usepackage{mathtools}
\usepackage{authblk}

\usepackage{tikz}

\def\mono#1#2{\filldraw (#1+.5,#2+.5) circle (2pt);}
\def\hdom#1#2{\filldraw[thick] (#1+.5,#2+.5) circle (2pt) --(#1+1.5,#2+.5) circle (2pt);}
\def\htri#1#2{\filldraw[thick](#1+.5,#2+.5) circle (2pt)--(#1+2.5,#2+.5) circle (2pt);}
\def\hquad#1#2{\filldraw[thick](#1+.5,#2+.5) circle (2pt)--(#1+3.5,#2+.5) circle (2pt);}

\def\hhex#1#2{\filldraw[thick](#1+.5,#2+.5) circle (2pt)--(#1+5.5,#2+.5) circle (2pt);}

\def\htwelve#1#2{\filldraw[thick](#1+.5,#2+.5) circle (2pt)--(#1+11.5,#2+.5) circle (2pt);}

\def\s{s_1}
\def\t{s_2}

\def\hsk{\hskip .2truein}
\def\subn{_r}

\def\subeight{_{8}}
\def\subthree{_{3}}
\def\suboneell{_{\mathcal{R}}}
\def\Alpha{s_r}
\def\Beta{s_{2r}}
\newcommand\bimn{{\arraycolsep=1.4pt\left\{\begin{array}{c}m\\n\end{array}\right\}\subn}}
\def\CombAB{C\subn(m,n)}
\def\Combmn{C_\mathcal{L}(m,n)}

\def\CatAB{\Cat_{r}(m)}
\def\rmT{\textrm{T}}
\def\ZS{\mathbb{Z}[S]}
\def\ello{\ell}

\def\kappan{\eta}
\def\etak{\rho}
\def\omega{\wt}
\def\Z{\mathbb{Z}}

\def\bim#1#2{{\arraycolsep=1.4pt\left\{\begin{array}{c}#1\\#2\end{array}\right\}}}

\theoremstyle{plain}
\newtheorem{theorem}{Theorem}
\newtheorem{lemma}{Lemma}
\newtheorem{Example}{Example}

\theoremstyle{definition}

\theoremstyle{remark}

\newcommand{\ds}{\displaystyle}

\DeclareMathOperator{\wt}{wt}
\DeclareMathOperator{\Cat}{Cat}
\DeclareMathOperator{\ddiv}{\textrm{div}}

\title{Multivariable Lucas Polynomials and Lucanomials}
\author{Edward E. Allen}
\author{Katherine Riley}
\author{Michael Weselcouch}
\affil{Department of Mathematics and Statistics, Wake Forest University, Winston-Salem, NC 27109}

\date{\today}

\begin{document}

\maketitle

\begin{abstract}
 \emph{Lucas polynomials} are polynomials in $\s$ and $\t$ defined recursively by $\{0\}=0$, $\{1\}=1$, and $\{m\}=\s\{m-1\}+\t\{m-2\}$ for $m \geq 2$.  
We generalize Lucas polynomials from 2-variable polynomials to multivariable polynomials.
This is done by first
defining $r$-Lucas polynomials $\{m\}\subn$ in the variables $s_1$, $s_r$, and $s_{2r}$.  We show that the binomial analogues of the $r$-Lucas polynomials are polynomial and give a combinatorial interpretation for them.  
We then extend the generalization of Lucas polynomials to an arbitrarily large set of variables.
Recursively defined generating functions are given for these multivariable Lucas polynomials.
We conclude by giving additional applications and insights.

\end{abstract}

\section{Introduction}

Let  $\mathbb{Z}[\s,\t]$ denote the set of polynomials in variables $\s$ and $\t$ with coefficients from the integers $\mathbb{Z}.$
Lucas recursively defined a set of polynomials in $\mathbb{Z}[\s,\t]$, referred to as \emph{Lucas polynomials}, by $\{0\}=0$, $\{1\}=1,$ and $\{m\}=\s\{m-1\}+\t\{m-2\}$
\cite{Lucas1878a,Lucas1878b,Lucas1878c}.
These polynomials generalize many well-known important
mathematical and combinatorial numbers.
It is easy to see that the corresponding generating function of the Lucas polynomials is given by 
\begin{equation}
L(x)=\sum_{k=0}^\infty \{k\}x^k=\frac{x}{1-s_1x-s_2x^2}.\label{D:LucasGenFunct}\end{equation}
Thus, any theorem about the Lucas polynomials
proves a corresponding theorem about any series whose generating function is given in (\ref{D:LucasGenFunct}) for specific values of $\s$ and $\t$.
If we
set $\s=\t=1$, the Lucas polynomials  reduce to the Fibonnacci sequence.  
With $\s=1+q$ and $\t=-q$, the Lucas polynomials reduce to the $q$-analogue of $n$,  %
$\{n\}= [n]_q=1+q+q^2+q^3+\cdots+q^{n-1}.$
Setting $s_1 = 1$ and $s_2 = 2$ generates the Pell numbers \cite{OEIS3}.
When $s_1 = 2$ and $s_2 = -1$, the integers are generated.

Using the Lucas polynomials as substitutions
for integers in combinatorial identities
has led to surprising analogues.
For example, substituting the the Lucas polynomials
into the formula for binomial coefficients yields
the rational functions--called \emph{Lucanomials}-- 
\begin{equation*}
\Combmn=
\bim{m}{n}
=
\frac{\{m\}!}{\{m-n\}!\{n\}!},
\end{equation*}
where $\ds{\{m\}!= \{m\}\cdot\{m-1\}!}$ for $m \ge 1$ and $\{0\}! = 1$.
Several authors have given combinatorial interpretations to the Lucanomials \cite{Gessel85, Sagan2010, Sagan2018, Sagan2019} that show that they are  in fact polynomials.
In \cite{Sagan2018} a combinatorial interpretation using \textit{partial tiling paths} is given and the Catalan analogues of the
Lucas polynomials are shown to be polynomial.

In this paper, we generalize Lucanomials from polynomials in two variables to polynomials in many variables.
This is done by first
defining polynomials $\{m\}\subn$ in terms of 
certain tiling words
in Section 2.
These polynomials have variables $s_1$, $\Alpha$, and $\Beta$ where $s_1$, $s_r$, and $s_{2r}$ correspond to tiles of length $1$, $r$, and $2r$, respectively.
We define rational functions 
\begin{equation}\CombAB=\bimn=\frac{\{m\}\subn!}
{\{m-n\}\subn! \{n\}\subn!}
\end{equation}
analogously to binomials.
The proof that $\CombAB$ is polynomial
is given in Section 4 by constructing a bijection $s_1^{\gamma}\cdot\phi_r$, for an appropriate value of $\gamma$,
that maps $\mathbb{Z}[\s,\t]$ into $\Z[s_1,s_2,\dots,]$. 
Essentially, each $\CombAB$ is the image of a product of $r$ Lucanomials under the map $s_1^{\gamma}\cdot\phi_r$.
In Section 5, we give a combinatorial interpretation of the polynomial $\CombAB$.
In Section 6, we give simple conditions  for when the Catalan analogues
\begin{equation}
\CatAB=\frac{1}{\{m+1\}\subn}\bim{2m}{m}\subn
\end{equation}
are polynomial.
In Section \ref{S:ToInfinity}, we show how to generalize Lucanomials to
an arbitrary set of variables $S=\{s_1, s_2, \dots\}$
and give conditions under which these are polynomial.
We also give recursively-defined generating functions for these new polynomials enabling anybody with a computer algebra system to explore this space using the analogous Taylor series command.
Finally, in Section 8, we give some additional applications and insights into these polynomials.
We should note that while the techniques described in this paper are applied to Lucas polynomials, they are applicable to any set of polynomials whose binomial analogues are polynomial.

\section{Tilings}

Set
$S=\{s_1,s_2, \dots\}$ and
$\ZS=\mathbb{Z}[s_1,s_2, \dots].$
For the purposes of this paper, we assume that $r$ is a positive integer.
Let $\tau=\{\tau_1,\tau_2,\dots\}$ be a collection of tiles
in which $\tau_i$ has length $i$.
A \emph{tiling word of $k$} 
is a word 
$\rmT=\tau_{i_1}\tau_{i_2}\cdots\tau_{i_j}$ 
consisting of an ordered sequence of entries from $\tau$ such that 
$k=i_1+i_2+\cdots+i_j.$
For example, line (\ref{F:gentiling}) gives a tiling $\rmT$ of $10$ cells
given by $\rmT=\tau_2\tau_3\tau_1^2\tau_3.$
\begin{equation}
\begin{tikzpicture}[scale=0.3, every node/.style={scale=0.7}]
\draw[-](0,0) rectangle  (1,1); 
\draw[-](1,0) rectangle  (2,1); 
\draw[-](2,0) rectangle  (3,1); 
\draw[-](3,0) rectangle  (4,1); 
\draw[-](4,0) rectangle  (5,1); 
\draw[-](5,0) rectangle  (6,1); 
\draw[-](6,0) rectangle  (7,1); 
\draw[-](7,0) rectangle  (8,1); 
\draw[-](8,0) rectangle  (9,1); 
\draw[-](9,0) rectangle  (10,1); 
\hdom{0}{0}
\htri{2}{0}
\mono{5}{0}
\mono{6}{0}
\htri{7}{0}
\node[draw=none] at (0,-0.5) {0};
\node[draw=none] at (1,-0.5) {1};
\node[draw=none] at (2,-0.5) {2};
\node[draw=none] at (3,-0.5) {3};
\node[draw=none] at (4,-0.5) {4};
\node[draw=none] at (5,-0.5) {5};
\node[draw=none] at (6,-0.5) {6};
\node[draw=none] at (7,-0.5) {7};
\node[draw=none] at (8,-0.5) {8};
\node[draw=none] at (9,-0.5) {9};
\node[draw=none] at (10,-0.5) {10};
\end{tikzpicture}
\label{F:gentiling}
\end{equation}
Define the \emph{weight of $T$} by
$$\omega(\rmT)=\prod_{\tau_i\in T} s_i^{\hbox{\# of tiles of length $i$}}.
$$
Therefore, with $\rmT=\tau_2\tau_3\tau_1^2\tau_3$ as in line (\ref{F:gentiling}), we have $\omega(\rmT)=s_1^2s_2s_3^2$.
While the variables in $S$ are commutative, the variables in $\tau$ are not.

Recall that the Lucas polynomials are polynomials in $\mathbb{Z}[\s,\t]$
 defined by
$\{0\}=0$, $\{1\}=1$ and $\{m\}=s_1\cdot\{m-1\}+s_2\cdot\{m-2\}.$
It is not difficult to show that
\[\{m+1\}=\sum_{\rmT\in\Delta_{m}}\omega(T),\]
where $\Delta_{m}$ is the collection of tiling words of $m$
with monominoes ($\tau_1$'s) and dominoes ($\tau_2$'s).
Lucas polynomials are known to satisfy the recurrence relation \cite{Sagan2010}
\begin{equation}
\{m\}=\{k\}\cdot\{m-k+1\}+s_2\cdot\{k-1\}\cdot\{m-k\}.
\end{equation}
Define the Lucanomial
\begin{equation}
\bim{m}{n} = \frac{\{m\}!}{\{m -n\}!\{n\}!}.
\end{equation}
Thus, the Lucanomials have the recurrence relation
\begin{equation}
\bim{m}{n}=\{m-n+1\}\cdot\bim{m-1}{n-1}+s_2\cdot\{n-1\}\cdot\bim{m-1}{n}.
\end{equation}

\begin{Example}
With $\bim{m}{1}=\{m\}$,
 it is not hard to see that
$\bim{3}{2}=s_1^2+s_2$
and
$\bim{4}{2}=s_1^4+3s_1^2s_2+2s_2^2$.
\end{Example}

Let $\Delta_{m,r}$ be the collection of tiling words of $m$
that begin with $m \pmod{r}$ $\tau_1$ tiles
and then uses tiles from $\{\tau_r,\tau_{2r}\}$
indiscriminately. Note that our use of $\pmod{r}$ is as a function with codomain $\{0, 1, \dots,  r-1\}$, not as a relation.
Set the \emph{$r$-Lucas polynomial} $\{m\}\subn$ to be
\begin{equation}
\{m+1\}\subn=\sum_{\rmT\in \Delta_{m,r}}\omega(\rmT),
\end{equation}
with $\{0\}\subn=0$ and $\{1\}\subn=1$.  

\begin{Example}
Computationally, with $r=3$ and $m=11$,
\begin{equation}
\{
11\}_3
=
s_1s_3^3+2s_1s_3s_6.\notag
\end{equation}
This correspond to the following tilings:

    \centering
\begin{tikzpicture}[scale=0.3, every node/.style={scale=0.5}]
\draw[-](0,0) rectangle  (1,1); 
\draw[-](1,0) rectangle  (2,1); 
\draw[-](2,0) rectangle  (3,1); 
\draw[-](3,0) rectangle  (4,1); 
\draw[-](4,0) rectangle  (5,1); 
\draw[-](5,0) rectangle  (6,1); 
\draw[-](6,0) rectangle  (7,1); 
\draw[-](7,0) rectangle  (8,1); 
\draw[-](8,0) rectangle  (9,1); 
\draw[-](9,0) rectangle  (10,1); 
\mono{0}{0}
\htri{1}{0}
\htri{4}{0}
\htri{7}{0}
\end{tikzpicture}
\hsk
\begin{tikzpicture}[scale=0.3, every node/.style={scale=0.5}]
\draw[-](0,0) rectangle  (1,1); 
\draw[-](1,0) rectangle  (2,1); 
\draw[-](2,0) rectangle  (3,1); 
\draw[-](3,0) rectangle  (4,1); 
\draw[-](4,0) rectangle  (5,1); 
\draw[-](5,0) rectangle  (6,1); 
\draw[-](6,0) rectangle  (7,1); 
\draw[-](7,0) rectangle  (8,1); 
\draw[-](8,0) rectangle  (9,1); 
\draw[-](9,0) rectangle  (10,1); 
\mono{0}{0}
\hhex{1}{0}
\htri{7}{0}
\end{tikzpicture}
\hsk
\begin{tikzpicture}[scale=0.3, every node/.style={scale=0.5}]
\draw[-](0,0) rectangle  (1,1); 
\draw[-](1,0) rectangle  (2,1); 
\draw[-](2,0) rectangle  (3,1); 
\draw[-](3,0) rectangle  (4,1); 
\draw[-](4,0) rectangle  (5,1); 
\draw[-](5,0) rectangle  (6,1); 
\draw[-](6,0) rectangle  (7,1); 
\draw[-](7,0) rectangle  (8,1); 
\draw[-](8,0) rectangle  (9,1); 
\draw[-](9,0) rectangle  (10,1); 
\mono{0}{0}
\htri{1}{0}
\hhex{4}{0}
\end{tikzpicture}
\end{Example}

The $r$-Lucas polynomials satisfy a similar recurrence relation to that of the Lucas polynomials.

\begin{lemma}
We have
$\{0\}\subn=0$, $\{i\}\subn=s_1^{i-1}$ for $1 \le i \le r$ and for $m > r$,
\begin{equation}
\{m\}\subn=\{m-r\}\subn\cdot  \Alpha
+\{m-2r\}\subn\cdot \Beta.
\end{equation}
Furthermore, the generating function for $\{m\}\subn$ is given by
\begin{equation}
 L_r(x)=\frac{x+s_1^1 x^2 + \cdots +s_1^{r-1} x^{r}}{1-\Alpha x^r - \Beta x^{2r}}.  
 \label{E:Recursion}
\end{equation}
\label{T:Recursion}
\end{lemma}

\begin{proof}
Let $\rmT\in \Delta_{m-1, r}$ where $m>r.$ It is either the case that $\rmT$ ends with the tile $\tau_r$ or the tile $\tau_{2r}.$
The collection of all words fulfilling the first case is enumerated by $\{m-r\}\subn \cdot s_r$ and the second by $\{m-2r\}\subn \cdot s_{2r}$.
The recursion immediately implies the generating function.
\end{proof}

From the definition of $\{m\}\subn$, one can see that $\{m\}\subn \in \ZS$.
Define the \emph{$r$-Lucanomial} by
\begin{equation}\CombAB=\bimn=\frac{\{m\}\subn!}
{\{m-n\}\subn! \{n\}\subn!}.
\end{equation}
We will show that these are in fact always polynomials.  It is these polynomials--and their generalizations--that
will be the focus of the remainder of paper.

\section{A useful formula}

With $0 \leq n \leq m$, define $\gamma_{r}(m,n)$ by
\[
\gamma_{r}(m,n)
=\sum_{i=1}^m ((i-1)\hskip -6pt\pmod{r})-\sum_{i=1}^n ((i-1)\hskip -6pt\pmod{r})
-\sum_{i=1}^{m-n} ((i-1)\hskip -6pt\pmod{r}).
\]
Let $a$ and $h$ be defined by $a=n\pmod{r}$ and $h=(m-n)\pmod{r}$.  It is not difficult to show that
\begin{align}
\gamma_{r}(m,n)
=
 &\begin{cases}
ah&\hbox{if $0\le h\le r-a$,}\\
(r-a)(r-h)&\hbox{if $r-a<h< r.$}
 \end{cases}   
 \label{E:Gamma}
\end{align}
 
\begin{Example}
With $r=11$, $m=37$, and $n=20$ we have $a=9$, $h=6$ and
 $$\gamma_{11}(37,20)=171-91-70=10=(11-9)(11-6).$$
\end{Example}

\section{The function $s_1^{\gamma}\cdot\phi_r$}

Construct the homomorphism $\phi_r:\mathbb{Z}[\s,\t]\rightarrow \ZS$ 
by multiplicatively and linearly extending 
\begin{equation}
\phi_r(x)=\begin{cases}
1&\hbox{ if $x=1$},\\
s_r&\hbox{ if $x=\s$},\\
s_{2r}&\hbox{ if $x=\t$}.
\end{cases}
\end{equation}
Since the map $\phi_r$ substitutes one set of variables with another we have the following lemma:
\begin{lemma}
Let $P(\s,\t),Q(\s,\t)\in \mathbb{Z}[\s,\t].$ 
If $$\frac{P(\s,\t)}{Q(\s,\t)}\in \mathbb{Z}[\s,\t],$$
then
\begin{equation*}
\phi_r\Biggl(\frac{P(\s,\t)}{Q(\s,\t)}\Biggr)=\frac{\phi_r\Bigl(P(\s,\t)\Bigr)}{\phi_r\Bigl(Q(\s,\t)\Bigr)}=\frac{P(s_r, s_{2r})}{Q(s_r, s_{2r})}\in \ZS.
\end{equation*}
\label{T:IsPoly}
\end{lemma}

Set
\begin{equation}
\varepsilon_r(m)= ((m-1)\hskip -6pt\pmod{r}) +1.
    \label{E:varepsilon}
\end{equation}
Observe that $\varepsilon_r(m) = m \pmod{r}$ unless $m \pmod{r} = 0$.  In that case, $\varepsilon_r(m) = r$.
If $d = (m-1)(\ddiv\ r)$, then $\lceil \frac{m}{r}\rceil=d+1.$
This leads to the following.

\begin{lemma}
Suppose
$\varepsilon=\varepsilon_r(m)$
and $d=(m-1) (\ddiv\ r)$, then
\[
\{m\}\subn
=
\{\varepsilon\}\subn\cdot\phi_r\Bigl(\Bigl\{\frac{m-\varepsilon}{r}+1\Bigr\}\Bigr)
=
s_1^{\varepsilon-1}\cdot\phi_r(\{d+1\})
=
s_1^{\varepsilon-1}\cdot\phi_r\Bigr(\Bigl\{\Bigl\lceil\frac{m}{r}\Bigr\rceil\Bigr\}\Bigr)
\]
and
\begin{align}
\{m\}\subn\{m-r\}\subn\cdots\{\varepsilon\}\subn=
&s_1^{(\varepsilon-1)(d+1)}\cdot
\phi_r\Bigl(\{d+1\}!\Bigr)\notag\\
=&s_1^{(\varepsilon-1)(d+1)}\cdot
\phi_r\Bigl(\Bigl\{\Bigl\lceil\frac{m}{r}\Bigr\rceil\Bigr\}!\Bigr).\notag
\end{align}
\label{T:preimage}
\end{lemma}

\begin{proof}
Recall that $\{m\}\subn$ is determined by the tiling words of $m-1$ that begin with $(m-1) \pmod{r}$ $\tau_1$ tiles
and then uses tiles from $\{\tau_r,\tau_{2r}\}$
indiscriminately.
Both $\tau_r$ and $\tau_{2r}$ have tiles that are lengths that are multiples of $r$.
Thus, each tiling word involved in the computation of $\{m\}\subn$ starts with $(\varepsilon -1)$ $\tau_1$ tiles. 
The remaining $d\cdot r$ cells
are tiled by $\tau_r$ and $\tau_{2r}$, 
in the same way that $d$ cells can be tiled by $\tau_1$  and $\tau_2$.
\end{proof}

\noindent
This immediately yields the next result.
\begin{lemma}
We have 
$$\{m\}\subn!=s_1^{h_{r}(m)}\cdot \phi_r(R(\s,\t)),$$
where 
$$R(\s,\t)=\prod_{j=0}^{r-1} \Bigl\{\Bigl\lceil \frac{m-j}{r}\Bigr\rceil\Bigr\}!,$$
and 
$$h_{r}(m)=\sum_{i=1}^{m}( (i-1)\hskip -6pt\pmod{r}).
$$
\label{T:applyphi}
\end{lemma}


It follows immediately from Lemma \ref{T:applyphi}, that the power of $s_1$ in $\CombAB$ is given by $\gamma_r(m, n)$.

\begin{theorem}
Suppose $m=n+k.$  Let $M$, $N$ and $K$ be the multisets
\[M=\Bigl\{\Bigl\lceil \frac{m}{r} \Bigr\rceil,\Bigl\lceil \frac{m-1}{r} \Bigr\rceil, \ldots, \Bigl\lceil \frac{m-r+1}{r} \Bigr\rceil\Bigr\},\]  
\[N=\Bigl\{\Bigl\lceil \frac{n}{r} \Bigr\rceil,\Bigl\lceil \frac{n-1}{r} \Bigr\rceil, \ldots, \Bigl\lceil \frac{n-r+1}{r} \Bigr\rceil\Bigr\},\]
and
\[K=\Bigl\{\Bigl\lceil \frac{k}{r} \Bigr\rceil,\Bigl\lceil \frac{k-1}{r} \Bigr\rceil, \ldots, \Bigl\lceil \frac{k-r+1}{r} \Bigr\rceil\Bigr\}.\]
Then there is some ordering 
$M^*=[M_1,M_2,\cdots, M_r]$,
$N^*=[N_1,N_2,\cdots, N_r]$, and
$K^*=[K_1,K_2,\cdots, K_r]$ 
of the elements of the multisets $M, N,$ and $K$, respectively,
 such that
$M_i=N_i+K_i$, for $1 \le i \le r.$
\label{T:Matching}
\end{theorem}

\begin{proof}
Suppose $m=n+k$, $\mu = m\pmod{r}$,
$\alpha = n\pmod{r}$, and $\beta = k\pmod{r}$.
We break the proof of this theorem into cases dependent on whether or not the following are true: $\mu>0$ (if not $\mu=0$), $\alpha+\beta=\mu$ (if not, $\alpha+\beta=r+\mu$), and $\beta>0$ (if not $\beta=0$).
Of the eight cases, three contain contradictory conditions: 
$\mu=0$, $\alpha+\beta=\mu$, $\beta>0$;
$\mu=0$, $\alpha + \beta = r + \mu $, $\beta=0$; and 
$\mu>0$, $\alpha + \beta = r + \mu$, $\beta=0$.

We will create a 3-line array with entries of $M$ on the first row,
entries from $N$ on the second row, and entries from $K$ on the third
such that the entries in the second and third rows of a given column sum to the entry in the first row.  To do this, list the entries of $M$ and $K$ in weakly decreasing order from left to right.  In the second row, list the entries of $N$ in weakly increasing order starting in column $\mu+1$ and then continuing in the first column by wrapping around.

In Row 1, if $\mu=0$, then all of the entries are equal to $\frac{m}{r}.$
Otherwise, there are $\mu$ entries of $\frac{m-\mu+r}{r}$
and $r-\mu$ entries of $\frac{m-\mu}{r}.$
Recall that the entries in Row 2 are listed in increasing order starting in column $\mu+1.$
Thus, in Row 2, the first $r-\alpha$ entries (starting in column $\mu+1$) are $\frac{n-\alpha}{r}$
 followed by $\alpha$ entries of $\frac{n-\alpha+r}{r}$ (with wrap around).
In Row 3, the first $\beta$ entries are $\frac{k-\beta+r}{r}$
 followed by $r-\beta$ entries of $\frac{k-\beta}{r}$.
 In all five of the following cases, we will have $M_j=N_j+K_j$ for $1\le j \le r.$
 
\textbf{Case 1:} $\mu>0$, $\alpha+\beta=\mu$, $\beta>0$.

\noindent
Set $\theta=\frac{m-\mu+r}{r}$,
$\kappan=\frac{n-\alpha}{r}$, 
and 
$\etak=\frac{k-\beta+r}{r}.$
Note that $M_j=\theta$ for $1 \le j \le \mu$ and $M_j=\theta-1$ for $\mu+1 \le j \le r.$
Also, since $\beta=\mu-\alpha$, 
$N_j=\kappan-1$ for $1 \le j \le \beta$ and
$\mu+1 \le j \le r$ as well as
$N_j=\kappan$ for $\beta+1 \le j \le \mu$.
Finally,
$K_j=\etak$ for $1 \le j \le \beta$
and
$K_j=\etak-1$ for $\beta+1 \le j \le r.$

\textbf{Case 2:} $\mu>0$, $\alpha+\beta=r+\mu$, $\beta>0.$

\noindent
Set $\theta=\frac{m-\mu+r}{r}$,
$\kappan=\frac{n-\alpha+r}{r}$, 
and 
$\etak=\frac{k-\beta+r}{r}.$
Note that $M_j=\theta$ for $1 \le j \le \mu$ and $M_j=\theta-1$ for $\mu+1 \le j \le r.$
Also, since $\beta=\mu+(r-\alpha)$, we have
$N_j=\kappan$ for $1 \le j \le \mu$ and $\beta+1\le j \le r$
as well as $N_j=\kappan-1$ for 
$\mu+1 \le j \le \beta=\mu+r-\alpha$.
Finally,
$K_j=\etak$ for $1 \le j \le \beta$
and
$K_j=\etak-1$ for $\beta+1 \le j \le r.$

\textbf{Case 3:} $\mu=0$, $\alpha+\beta=r+\mu=r$, $\beta>0.$

\noindent
Set $\theta=\frac{m}{r}$,
$\kappan=\frac{n-\alpha+r}{r}$, 
and 
$\etak=\frac{k-\beta+r}{r}.$
Note that $M_j=\theta$ for $1 \le j \le r$.
Also, since $\beta=r-\alpha$, we have
$N_j=\kappan-1$ for $1 \le j \le \beta$ 
as well as $N_j=\kappan$ for 
$\beta+1 \le j \le r$.
Finally,
$K_j=\etak$ for $1 \le j \le \beta$
and
$K_j=\etak-1$ for $\beta+1 \le j \le r.$

\textbf{Case 4:} $\mu>0$, $\alpha=\alpha+\beta=\mu$, $\beta=0.$

\noindent
Set $\theta=\frac{m-\mu+r}{r}$,
$\kappan=\frac{n-\alpha+r}{r}$, 
and 
$\etak=\frac{k}{r}.$
Note that $M_j=\theta$ for $1 \le j \le \mu$ and $M_j=\theta-1$ for $\mu+1 \le j \le r.$
Also, since $\alpha=\mu$, we have
$N_j=\kappan$ for $1 \le j \le \mu$ 
as well as $N_j=\kappan-1$ for 
$\mu+1 \le j \le r$.
Finally,
$K_j=\etak$ for $1 \le j \le r$.

\textbf{Case 5:} $\mu=0$, $\alpha=0$, $\beta=0.$

\noindent
Set $\theta=\frac{m}{r}$,
$\kappan=\frac{n}{r}$, 
and 
$\etak=\frac{k}{r}.$
Note that $M_j=\theta$ for $1 \le j \le r$,
$N_j=\kappan$ for $1 \le j \le r$ 
and
$K_j=\etak$ for $1 \le j \le r$.
\end{proof}

Example \ref{Ex:prodofbims} illustrates the first case of the proof of Theorem \ref{T:Matching}.

\begin{Example}
Suppose $m=37$, $n=10$. $m-n=27$ and $r=8.$ It follows that
\[\{37\}_8!=s_1^{122}\cdot\phi_8\Bigl(\{5\}!\cdot \{5\}!\cdot\{5\}!\cdot\{5\}!\cdot\{5\}!\cdot\{4\}!\cdot\{4\}!\cdot\{4\}!\Bigr),\]
\[\{27\}_8!=s_1^{87}\cdot\phi_8\Bigl(\{4\}!\cdot\{4\}!\cdot\{4\}!\cdot\{3\}!\cdot\{3\}!\cdot\{3\}!\cdot\{3\}!\cdot\{3\}!\Bigr),\] and
\[\{10\}_8!=s_1^{29}\cdot\phi_8\Bigl(\{2\}!\cdot\{2\}!\cdot\{1\}!\cdot\{1\}!\cdot\{1\}!\cdot\{1\}!\cdot\{1\}!\cdot\{1\}!\Bigr).\]
Now, write $55555444$ decreasing, $44433333$ increasing starting in column $1+(37 \pmod{8})=6$ and then continuing with wrap-around
and then $22111111$ decreasing.
$$\begin{array}{cccccccc}
5&5&5&5&5&4&4&4\\
3&3&4&4&4&3&3&3\\
2&2&1&1&1&1&1&1 
\end{array}$$
The columns add up correctly so that we can now create
corresponding
product of binomials:
\[
\bim{37}{27}\subeight = s_1^6 \cdot
\phi_8\Biggl(
\bim{5}{3}^2\cdot
\bim{5}{4}^3\cdot
\bim{4}{3}^3\Biggr)
\]
The exponent of $s_1$ is given by (\ref{E:Gamma})
with $a=27\pmod{8}=3$, $h=(37-27)\pmod{8}=2$ and $\gamma_{8}(37,27)=3\cdot2=6.$ 
\label{Ex:prodofbims}
\end{Example}

The previous computation, 
in which
$\CombAB$
is written as the image of
a product of Lucas polynomials under the map $s_1^{\gamma_{8}(37,27)}\cdot\phi_8$, is an example of a more general theorem. 
\begin{theorem}
For any $\CombAB$, there exists a product \[R(\s,\t)=\prod_{i=1}^r \bim{a_i}{b_i}\]
of Lucanomials
such that 
\begin{equation}
\CombAB=\bimn=s_1^{\gamma_{r}(m,n)}\cdot\phi_r\Bigl(R(\s,\t)\Bigr)
\end{equation}
where $\gamma_{r}(m,n)$ is given in (\ref{E:Gamma}).
Therefore, $\CombAB$ is a polynomial.
\label{T:Big2Theorem}
\end{theorem}

\begin{proof}
The product of binomials $R(\s,\t)$ is constructed in Theorem \ref{T:Matching} and
$\gamma_{r}(m,n)$, as shown in (\ref{E:Gamma}), solely counts the difference between the exponent of $s_1$ in $\{m\}\subn!$ and the exponent of $s_1$ in $\{m-n\}\subn!\cdot \{n\}\subn!$.  
Lemma \ref{T:IsPoly} implies the polynomiality of the result.
\end{proof}

Theorem \ref{T:Matching} and Theorem \ref{T:Big2Theorem} give another method for computing $\CombAB$ as shown in the following example.

\begin{Example}

Using the recurrence relation in Lemma \ref{T:Recursion} with $r=3$,
computationally
\begin{equation}
\Bigl\{
{\arraycolsep=1.4pt
\begin{array}{c}
10\\ 5
\end{array}}
\Bigr\}\subthree
=s_1s_3^8+5s_1s_3^6s_6+9s_1s_3^4s_6^2+7s_1s_3^2s_6^3+2s_1s_6^4.
\label{Ex:tenfivefive}
\end{equation}
Now,
$$\{10\}\subthree!=s_1^{9}\cdot\phi_3\Bigl(\{4\}\cdot\{3\}\cdot\{3\}\cdot\{3\}\cdot\{2\}\cdot\{2\}\cdot\{2\}\cdot\{1\}\cdot\{1\}\cdot\{1\}\Bigr),$$
$$\{5\}_3!=s_1^{4}\cdot\phi_3\Bigl(\{2\}\cdot\{2\}\cdot\{1\}\cdot\{1\}\cdot\{1\}\Bigr),$$
and thus, using Theorem \ref{T:Matching}, gives
\begin{align}
\frac{\left\{10\right\}\subthree!}{\left\{5\right\}\subthree!\cdot\left\{5\right\}\subthree!}= &
\frac{s_1^{9}\cdot\phi_3\Bigl(\{4\}\cdot\{3\}\cdot\{3\}\cdot\{3\}\cdot\{2\}\cdot\{2\}\cdot\{2\}\cdot\{1\}\cdot\{1\}\cdot\{1\}\Bigr)}
{s_1^{8}\cdot\phi_3\Bigl(\{2\}\cdot\{2\}\cdot\{1\}\cdot\{1\}\cdot\{1\}\cdot\{2\}\cdot\{2\}\cdot\{1\}\cdot\{1\}\cdot\{1\}\Bigr)}\notag
\\
=&s_1\cdot \frac{\phi_3\Bigl(\{4\}!\cdot\{3\}!\cdot\{3\}!)\Bigr)}
{\phi_3\Bigl((\{2\}!\cdot\{2\}!)\cdot(\{2\}!\cdot\{1\}!)\cdot(\{2\}!\cdot\{1\}!)\Bigr)}\notag\\
=&s_1\cdot\phi_3\Biggl(\left\{
{\arraycolsep=1.4pt
\begin{array}{c}4\\2\end{array}}
\right\}
\cdot\left\{
{\arraycolsep=1.4pt
\begin{array}{c}3\\2\end{array}}\right\}
\cdot\left\{
{\arraycolsep=1.4pt
\begin{array}{c}3\\2\end{array}}
\right\}\Biggr)
\label{E:binexp}\\
=&s_1\cdot\phi_3\Bigl((\s^4+3\s^2\t+2\t^2)\cdot(\s^2+\t)\cdot(\s^2+\t)\Bigr)\notag\\
=&s_1\cdot\phi_3\Bigl(\s^{8}+5\s^6\t+9\s^4\t^2+7\s^2\t^3+2\t^4\Bigr).
\label{E:Ideaworks}
\end{align}
The action of $\phi_3$, which maps $\s\rightarrow s_3$ and $\t \rightarrow s_6$, in
(\ref{E:Ideaworks}) yields (\ref{Ex:tenfivefive}). 
Finally, we have $\gamma_{3}(10,5)=(3-2)(3-2)=1$ since $a=5\pmod{3}=2$ and $h=(10-5)\pmod{3}=2$. 
\end{Example}

\section{Combinatorial interpretation for $\CombAB$}
\label{S:CombInterp}
In \cite{Sagan2018}, the authors give a combinatorial interpretation for $\Combmn$. 
This interpretation considers certain lattice paths inside tilings of Young diagrams.  The combinatorial interpretation for $\CombAB$ will consider sequences of lattice paths inside tilings of Young diagrams.

Let $\lambda = (\lambda_1, \lambda_2, \dots, \lambda_l)$ be an integer partition, that is, $\lambda_1 \geq \lambda_2 \geq \dots \geq \lambda_l >0$.  We say that each $\lambda_i$ is a \emph{part} of $\lambda$ and the \emph{length}, $l(\lambda)$, of $\lambda$ is the number of its parts.  The \emph{Young diagram} of $\lambda$ is an array of left-justified rows of boxes which we will write in French notation so that $\lambda_i$ is the number of boxes in the $i^{\text{th}}$,  row from the bottom of the diagram.  We embed the Young diagram of $\lambda$ in the Cartesian plane in the following way: the bottom left corner of the leftmost box of the first row is at the point $(0, 0)$, and each box is a unit square.  For a reason that will soon become clear, we will include the line segment from $(\lambda_1, 0)$ to $(\lambda_1+1, 0)$ and the line segment from $(0, l(\lambda))$ to $(0, l(\lambda)+1)$ as part of the embedding of the Young diagram of $\lambda$ in the Cartesian plane. We will use the notation $\lambda$ to denote the Young diagram of $\lambda$.

Define $r\cdot\lambda$ by
\[ r\cdot\lambda = (r\lambda_1, r\lambda_2, \dots, r\lambda_l).\]

\noindent
An important partition in this section is the staircase partition $\delta_m$, defined by
\[ \delta_m = (m-1, m-2, \dots, 1).\]


A \emph{tiling} of $\lambda$ is a sequence $(w_1, w_2, \dots, w_{l(\lambda)})$ where for all $i$, $w_i$ is a tiling word of $\lambda_i$ by $\tau_1$ and $\tau_2$.  This can be represented by tiling the cells of the rows of $\lambda$ as seen in Example \ref{E: delta7}.  Let $\mathcal{T}(\lambda)$ be the set of all such tilings of $\lambda$ and let
\[\wt(\lambda) = \sum_{T \in\mathcal{T}(\lambda)}\wt(T).\]
It is easy to see that $\wt(\delta_m) = \{m\}!$.  The authors of \cite{Sagan2018} show that 
$\Combmn$
is a polynomial by partitioning $\mathcal{T}(\delta_n)$ into blocks $\beta_1, \beta_2, \dots, \beta_k$, where $\{m-n\}!\{n\}!$ divides $\wt(\beta_i)$ for all $i$ and for all $0\leq n \leq m$.

\begin{Example} \label{E: delta7}
The following is the embedding of $\delta_7$ in the Cartesian plane and a tiling $T$ of $\delta_7$.  It follows that $\wt(T) = s_1^9+s_2^6$.

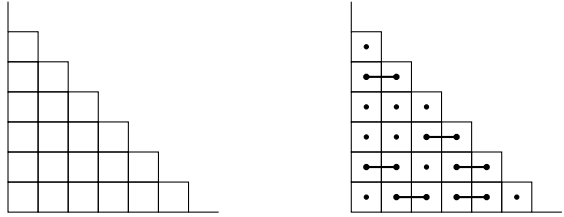
\begin{center}
\begin{tikzpicture}[scale=0.4, every node/.style={scale=0.5}]
\draw[-](0,0) rectangle  (1,1); 
\draw[-](1,0) rectangle  (2,1); 
\draw[-](2,0) rectangle  (3,1); 
\draw[-](3,0) rectangle  (4,1); 
\draw[-](4,0) rectangle  (5,1); 
\draw[-](5,0) rectangle  (6,1); 

\draw[-](0,1) rectangle  (1,2); 
\draw[-](1,1) rectangle  (2,2); 
\draw[-](2,1) rectangle  (3,2); 
\draw[-](3,1) rectangle  (4,2); 
\draw[-](4,1) rectangle  (5,2); 

\draw[-](0,2) rectangle  (1,3); 
\draw[-](1,2) rectangle  (2,3); 
\draw[-](2,2) rectangle  (3,3); 
\draw[-](3,2) rectangle  (4,3); 

\draw[-](0,3) rectangle  (1,4); 
\draw[-](1,3) rectangle  (2,4); 
\draw[-](2,3) rectangle  (3,4);

\draw[-](0,4) rectangle  (1,5); 
\draw[-](1,4) rectangle  (2,5);

\draw[-](0,5) rectangle  (1,6); 

\draw[-] (6, 0) to (7, 0);
\draw[-] (0, 6) to (0, 7);
\end{tikzpicture}
\hspace{1.5 cm}
\begin{tikzpicture}[scale=0.4, every node/.style={scale=0.5}]
\draw[-](0,0) rectangle  (1,1); 
\draw[-](1,0) rectangle  (2,1); 
\draw[-](2,0) rectangle  (3,1); 
\draw[-](3,0) rectangle  (4,1); 
\draw[-](4,0) rectangle  (5,1); 
\draw[-](5,0) rectangle  (6,1); 
\mono{0}{0}
\hdom{1}{0}
\hdom{3}{0}
\mono{5}{0}

\draw[-](0,1) rectangle  (1,2); 
\draw[-](1,1) rectangle  (2,2); 
\draw[-](2,1) rectangle  (3,2); 
\draw[-](3,1) rectangle  (4,2); 
\draw[-](4,1) rectangle  (5,2); 
\hdom{0}{1}
\mono{2}{1}
\hdom{3}{1}

\draw[-](0,2) rectangle  (1,3); 
\draw[-](1,2) rectangle  (2,3); 
\draw[-](2,2) rectangle  (3,3); 
\draw[-](3,2) rectangle  (4,3);
\mono{0}{2}
\mono{1}{2}
\hdom{2}{2}

\draw[-](0,3) rectangle  (1,4); 
\draw[-](1,3) rectangle  (2,4); 
\draw[-](2,3) rectangle  (3,4);
\mono{0}{3}
\mono{1}{3}
\mono{2}{3}

\draw[-](0,4) rectangle  (1,5); 
\draw[-](1,4) rectangle  (2,5);
\hdom{0}{4}

\draw[-](0,5) rectangle  (1,6); 
\mono{0}{5}

\draw[-] (6, 0) to (7, 0);
\draw[-] (0, 6) to (0, 7);

\end{tikzpicture}
\end{center}
\end{Example}

A \emph{partial tiling path of $\delta_m$} is a lattice path $p$ from $(n, 0)$ to $(0, m)$ consisting of unit steps north (N) and west (W) that have the following properties: $p$ is contained in $\delta_m$, every W step must be followed by an N step, and if $p$ contains the point $(m-i, i)$ for some $i$, the step from that point is an N.

Let $p$ be a partial tiling path of $\delta_m$ from $(n, 0)$ to $(0, m)$. A \emph{binomial partial tiling of type $p$} is a partial tiling of the rows of $\delta_m$ such that if the step N is immediately preceded by W, the tiling is of the cells to the right of that N step and the first tile must be $\tau_2$ if the number of cells is greater than $0$.  Otherwise, tile the cells in the part of the row left of the N step with $\tau_1$'s and $\tau_2$'s.

\begin{Example}
The following is a partial tiling path $p$ of $\delta_7$ and a binomial partial tiling of type $p$.  In this case, $n=4$.

\begin{center}
\begin{tikzpicture}[scale=0.4, every node/.style={scale=0.5}]
\draw[-](0,0) rectangle  (1,1); 
\draw[-](1,0) rectangle  (2,1); 
\draw[-](2,0) rectangle  (3,1); 
\draw[-](3,0) rectangle  (4,1); 
\draw[-](4,0) rectangle  (5,1); 
\draw[-](5,0) rectangle  (6,1); 

\draw[-](0,1) rectangle  (1,2); 
\draw[-](1,1) rectangle  (2,2); 
\draw[-](2,1) rectangle  (3,2); 
\draw[-](3,1) rectangle  (4,2); 
\draw[-](4,1) rectangle  (5,2); 

\draw[-](0,2) rectangle  (1,3); 
\draw[-](1,2) rectangle  (2,3); 
\draw[-](2,2) rectangle  (3,3); 
\draw[-](3,2) rectangle  (4,3); 

\draw[-](0,3) rectangle  (1,4); 
\draw[-](1,3) rectangle  (2,4); 
\draw[-](2,3) rectangle  (3,4);

\draw[-](0,4) rectangle  (1,5); 
\draw[-](1,4) rectangle  (2,5);

\draw[-](0,5) rectangle  (1,6); 

\draw[-] (6, 0) to (7, 0);
\draw[-] (0, 6) to (0, 7);

\draw[very thick, -] (4, 0) to (3, 0);
\draw[very thick, -] (3, 0) to (3, 1);
\draw[very thick, -] (3, 1) to (3, 2);
\draw[very thick, -] (3, 2) to (2, 2);
\draw[very thick, -] (2, 2) to (2, 3);
\draw[very thick, -] (2, 3) to (2, 4);
\draw[very thick, -] (2, 4) to (2, 5);
\draw[very thick, -] (2, 5) to (1, 5);
\draw[very thick, -] (1, 5) to (1, 6);
\draw[very thick, -] (1, 6) to (0, 6);
\draw[very thick, -] (0, 6) to (0, 7);

\end{tikzpicture}
\hspace{1.5 cm}
\begin{tikzpicture}[scale=0.4, every node/.style={scale=0.5}]
\draw[-](0,0) rectangle  (1,1); 
\draw[-](1,0) rectangle  (2,1); 
\draw[-](2,0) rectangle  (3,1); 
\draw[-](3,0) rectangle  (4,1); 
\draw[-](4,0) rectangle  (5,1); 
\draw[-](5,0) rectangle  (6,1);

\draw[-](0,1) rectangle  (1,2); 
\draw[-](1,1) rectangle  (2,2); 
\draw[-](2,1) rectangle  (3,2); 
\draw[-](3,1) rectangle  (4,2); 
\draw[-](4,1) rectangle  (5,2); 

\draw[-](0,2) rectangle  (1,3); 
\draw[-](1,2) rectangle  (2,3); 
\draw[-](2,2) rectangle  (3,3); 
\draw[-](3,2) rectangle  (4,3); 

\draw[-](0,3) rectangle  (1,4); 
\draw[-](1,3) rectangle  (2,4); 
\draw[-](2,3) rectangle  (3,4);

\draw[-](0,4) rectangle  (1,5); 
\draw[-](1,4) rectangle  (2,5);

\draw[-](0,5) rectangle  (1,6); 

\draw[-] (6, 0) to (7, 0);
\draw[-] (0, 6) to (0, 7);

\draw[very thick, -] (4, 0) to (3, 0);
\draw[very thick, -] (3, 0) to (3, 1);
\draw[very thick, -] (3, 1) to (3, 2);
\draw[very thick, -] (3, 2) to (2, 2);
\draw[very thick, -] (2, 2) to (2, 3);
\draw[very thick, -] (2, 3) to (2, 4);
\draw[very thick, -] (2, 4) to (2, 5);
\draw[very thick, -] (2, 5) to (1, 5);
\draw[very thick, -] (1, 5) to (1, 6);
\draw[very thick, -] (1, 6) to (0, 6);
\draw[very thick, -] (0, 6) to (0, 7);

\hdom{3}{0}
\mono{5}{0}
\hdom{0}{1}
\mono{2}{1}
\hdom{2}{2}
\mono{0}{3}
\mono{1}{3}
\hdom{0}{4}

\end{tikzpicture}
\end{center}
\end{Example}

\begin{theorem}\cite{Sagan2018}
Given $0 \leq n \leq m$, we have 
\[ \bim{m}{n} = \sum_B\wt(B),\]
where the sum is over all binomial partial tilings associated with partial tiling paths of $\delta_m$ from $(n, 0)$ to $(0, m)$.
\end{theorem}

A \emph{$r$-partial tiling path of $r \cdot \delta_m$} is a lattice path $p$ from $(r\cdot n, 0)$ to $(0, m)$ consisting of unit steps north (N) and length $r$ steps west (W) that have the following properties: $p$ is contained in $r\cdot \delta_m$, every W step must be followed by an N step, and if $p$ contains the point $(r\cdot(m-i), i)$ for some $i$, the step from that point is an N.  It is easy to see that the set of $r$-partial tiling paths of $r \cdot \delta_m$ is in bijection with the set of partial tiling paths of $\delta_m$.

Let $p$ be a partial tiling path of $r\cdot \delta_m$ from $(r\cdot n, 0)$ to $(0, m)$. A \emph{$r$-binomial partial tiling of type $p$} is a partial tiling of the rows of $r\cdot \delta_m$ with $\tau_r$'s and $\tau_{2r}$'s such that if the step N is immediately preceded by W, the tiling is of the cells to the right of that N step and the first tile must be $\tau_{2r}$ if the number of cells is greater than $0$.  Otherwise, tile the cells to the left of N with $\tau_r$'s and $\tau_{2r}$'s.

The set of all $r$-binomial partial tilings is in bijection with set of all binomial partial tilings.

\begin{Example}
The following is a $3$-partial tiling path $p$ of $3\cdot \delta_7$ and a $3$-binomial partial tiling of type $p$.  In this case, $n=4.$

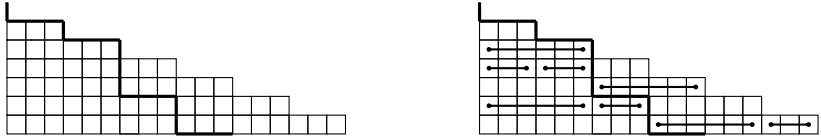
\begin{center}
\begin{tikzpicture}[scale=0.25, every node/.style={scale=0.5}]
\draw[-](0,0) rectangle  (1,1); 
\draw[-](1,0) rectangle  (2,1); 
\draw[-](2,0) rectangle  (3,1); 
\draw[-](3,0) rectangle  (4,1); 
\draw[-](4,0) rectangle  (5,1); 
\draw[-](5,0) rectangle  (6,1);
\draw[-](6,0) rectangle  (7,1); 
\draw[-](7,0) rectangle  (8,1); 
\draw[-](8,0) rectangle  (9,1); 
\draw[-](9,0) rectangle  (10,1); 
\draw[-](10,0) rectangle  (11,1); 
\draw[-](11,0) rectangle  (12,1);
\draw[-](12,0) rectangle  (13,1); 
\draw[-](13,0) rectangle  (14,1); 
\draw[-](14,0) rectangle  (15,1); 
\draw[-](15,0) rectangle  (16,1); 
\draw[-](16,0) rectangle  (17,1); 
\draw[-](17,0) rectangle  (18,1);

\draw[-](0,1) rectangle  (1,2); 
\draw[-](1,1) rectangle  (2,2); 
\draw[-](2,1) rectangle  (3,2); 
\draw[-](3,1) rectangle  (4,2); 
\draw[-](4,1) rectangle  (5,2);
\draw[-](5,1) rectangle  (6,2); 
\draw[-](6,1) rectangle  (7,2); 
\draw[-](7,1) rectangle  (8,2); 
\draw[-](8,1) rectangle  (9,2); 
\draw[-](9,1) rectangle  (10,2);
\draw[-](10,1) rectangle  (11,2); 
\draw[-](11,1) rectangle  (12,2); 
\draw[-](12,1) rectangle  (13,2); 
\draw[-](13,1) rectangle  (14,2); 
\draw[-](14,1) rectangle  (15,2);

\draw[-](0,2) rectangle  (1,3); 
\draw[-](1,2) rectangle  (2,3); 
\draw[-](2,2) rectangle  (3,3); 
\draw[-](3,2) rectangle  (4,3);
\draw[-](4,2) rectangle  (5,3); 
\draw[-](5,2) rectangle  (6,3); 
\draw[-](6,2) rectangle  (7,3); 
\draw[-](7,2) rectangle  (8,3);
\draw[-](8,2) rectangle  (9,3); 
\draw[-](9,2) rectangle  (10,3); 
\draw[-](10,2) rectangle  (11,3); 
\draw[-](11,2) rectangle  (12,3);

\draw[-](0,3) rectangle  (1,4); 
\draw[-](1,3) rectangle  (2,4); 
\draw[-](2,3) rectangle  (3,4);
\draw[-](3,3) rectangle  (4,4); 
\draw[-](4,3) rectangle  (5,4); 
\draw[-](5,3) rectangle  (6,4);
\draw[-](6,3) rectangle  (7,4); 
\draw[-](7,3) rectangle  (8,4); 
\draw[-](8,3) rectangle  (9,4);

\draw[-](0,4) rectangle  (1,5); 
\draw[-](1,4) rectangle  (2,5);
\draw[-](2,4) rectangle  (3,5); 
\draw[-](3,4) rectangle  (4,5);
\draw[-](4,4) rectangle  (5,5); 
\draw[-](5,4) rectangle  (6,5);

\draw[-](0,5) rectangle  (1,6);
\draw[-](1,5) rectangle  (2,6);
\draw[-](2,5) rectangle  (3,6);

\draw[-] (6, 0) to (7, 0);
\draw[-] (0, 6) to (0, 7);

\draw[very thick, -] (12, 0) to (9, 0);
\draw[very thick, -] (9, 0) to (9, 1);
\draw[very thick, -] (9, 1) to (9, 2);
\draw[very thick, -] (9, 2) to (6, 2);
\draw[very thick, -] (6, 2) to (6, 3);
\draw[very thick, -] (6, 3) to (6, 4);
\draw[very thick, -] (6, 4) to (6, 5);
\draw[very thick, -] (6, 5) to (3, 5);
\draw[very thick, -] (3, 5) to (3, 6);
\draw[very thick, -] (3, 6) to (0, 6);
\draw[very thick, -] (0, 6) to (0, 7);

\end{tikzpicture}
\hspace{1.5 cm}
\begin{tikzpicture}[scale=0.25, every node/.style={scale=0.5}]
\draw[-](0,0) rectangle  (1,1); 
\draw[-](1,0) rectangle  (2,1); 
\draw[-](2,0) rectangle  (3,1); 
\draw[-](3,0) rectangle  (4,1); 
\draw[-](4,0) rectangle  (5,1); 
\draw[-](5,0) rectangle  (6,1);
\draw[-](6,0) rectangle  (7,1); 
\draw[-](7,0) rectangle  (8,1); 
\draw[-](8,0) rectangle  (9,1); 
\draw[-](9,0) rectangle  (10,1); 
\draw[-](10,0) rectangle  (11,1); 
\draw[-](11,0) rectangle  (12,1);
\draw[-](12,0) rectangle  (13,1); 
\draw[-](13,0) rectangle  (14,1); 
\draw[-](14,0) rectangle  (15,1); 
\draw[-](15,0) rectangle  (16,1); 
\draw[-](16,0) rectangle  (17,1); 
\draw[-](17,0) rectangle  (18,1);

\draw[-](0,1) rectangle  (1,2); 
\draw[-](1,1) rectangle  (2,2); 
\draw[-](2,1) rectangle  (3,2); 
\draw[-](3,1) rectangle  (4,2); 
\draw[-](4,1) rectangle  (5,2);
\draw[-](5,1) rectangle  (6,2); 
\draw[-](6,1) rectangle  (7,2); 
\draw[-](7,1) rectangle  (8,2); 
\draw[-](8,1) rectangle  (9,2); 
\draw[-](9,1) rectangle  (10,2);
\draw[-](10,1) rectangle  (11,2); 
\draw[-](11,1) rectangle  (12,2); 
\draw[-](12,1) rectangle  (13,2); 
\draw[-](13,1) rectangle  (14,2); 
\draw[-](14,1) rectangle  (15,2);

\draw[-](0,2) rectangle  (1,3); 
\draw[-](1,2) rectangle  (2,3); 
\draw[-](2,2) rectangle  (3,3); 
\draw[-](3,2) rectangle  (4,3);
\draw[-](4,2) rectangle  (5,3); 
\draw[-](5,2) rectangle  (6,3); 
\draw[-](6,2) rectangle  (7,3); 
\draw[-](7,2) rectangle  (8,3);
\draw[-](8,2) rectangle  (9,3); 
\draw[-](9,2) rectangle  (10,3); 
\draw[-](10,2) rectangle  (11,3); 
\draw[-](11,2) rectangle  (12,3);

\draw[-](0,3) rectangle  (1,4); 
\draw[-](1,3) rectangle  (2,4); 
\draw[-](2,3) rectangle  (3,4);
\draw[-](3,3) rectangle  (4,4); 
\draw[-](4,3) rectangle  (5,4); 
\draw[-](5,3) rectangle  (6,4);
\draw[-](6,3) rectangle  (7,4); 
\draw[-](7,3) rectangle  (8,4); 
\draw[-](8,3) rectangle  (9,4);

\draw[-](0,4) rectangle  (1,5); 
\draw[-](1,4) rectangle  (2,5);
\draw[-](2,4) rectangle  (3,5); 
\draw[-](3,4) rectangle  (4,5);
\draw[-](4,4) rectangle  (5,5); 
\draw[-](5,4) rectangle  (6,5);

\draw[-](0,5) rectangle  (1,6);
\draw[-](1,5) rectangle  (2,6);
\draw[-](2,5) rectangle  (3,6);

\draw[-] (6, 0) to (7, 0);
\draw[-] (0, 6) to (0, 7);

\draw[very thick, -] (12, 0) to (9, 0);
\draw[very thick, -] (9, 0) to (9, 1);
\draw[very thick, -] (9, 1) to (9, 2);
\draw[very thick, -] (9, 2) to (6, 2);
\draw[very thick, -] (6, 2) to (6, 3);
\draw[very thick, -] (6, 3) to (6, 4);
\draw[very thick, -] (6, 4) to (6, 5);
\draw[very thick, -] (6, 5) to (3, 5);
\draw[very thick, -] (3, 5) to (3, 6);
\draw[very thick, -] (3, 6) to (0, 6);
\draw[very thick, -] (0, 6) to (0, 7);

\hhex{9}{0}
\htri{15}{0}

\hhex{0}{1}
\htri{6}{1}

\hhex{6}{2}

\htri{0}{3}
\htri{3}{3}

\hhex{0}{4}

\end{tikzpicture}
\end{center}
\end{Example}

\begin{theorem}\label{T:Staircase}
We have 
\[ \phi_r\biggl(\bim{m}{n}\biggr) = \sum_C\wt(C),\]
where the sum is over all $r$-binomial partial tilings associated with $r$-partial tiling paths of $r \cdot \delta_m$ from $(r\cdot n, 0)$ to $(0, m)$.
\end{theorem}

\begin{proof}
Let $\Phi_r$ be the map that sends a binomial partial tiling to its associated $r$-binomial partial tiling.  This map sends the tile $\tau_1$ to $\tau_r$ and sends the tile $\tau_2$ to $\tau_{2r}$.  Therefore, $\wt(\Phi_r(B)) = \phi_r(\wt(B))$. Since $\phi_r$ is a homomorphism, we have  
\[ \phi_r\biggl(\bim{m}{n}\biggr) =\phi_r\bigl(\sum_B\wt(B)\bigr)= \sum_B\wt(\Phi_r(B)),\]
where both sums are over all binomial partial tilings associated with partial tiling path of $\delta_m$ from $(n, 0)$ to $(0, m)$.

For every $r$-partial tiling $C$, there is a unique partial tiling $B$ such that $C = \Phi_r(B)$.  Therefore, we have 
\[ \phi_r\biggl(\bim{m}{n}\biggr) = \sum_C\wt(C),\]
where the sum is over all $r$-binomial partial tilings associated with $r$-partial tiling path of $r \cdot \delta_m$ from $(r\cdot n, 0)$ to $(0, m)$ as desired.
\end{proof}

A combinatorial interpretation for 
$\CombAB$ is found by combining Theorem \ref{T:Big2Theorem} and Theorem \ref{T:Staircase}.  







\section{Catalan Analogues}
In \cite{Sagan2018}, the authors showed that the Catalan analogues of Lucas polynomials, 
$$\frac{\{2m\}!}{\{m+1\}!\cdot\{m\}!}
=\frac{1}{\{m+1\}}\bim{2m}{m}$$ 
are polynomial.  We will now show that this result does not always extend to $$\CatAB =\frac{1}{\{m+1\}\subn}\bim{2m}{m}\subn,$$ but there are cases in which it does. 

\begin{Example}
Not all rational functions of the form
$$\frac{1}{\{m+1\}\subn}\bim{2m}{m}\subn$$
are polynomial.
With $r=5$, computationally it is not hard to show that
$$\frac{1}{\{10\}_5}\bim{18}{9}_5=
\frac{(s_5^2+2s_{10})^3\cdot(s_5^2+s_{10})^5}{s_1^3s_5}.$$
Note that
\begin{align}
\frac{1}{\{10\}_5}\bim{18}{9}_5=&\frac{s_1^{33}\cdot\phi_5\Bigl((\{4\}!)^3\cdot (\{3\}!)^2\Bigr)}
{s_1^{32+4}\cdot\phi_5\Bigl(\{2\}\cdot(\{2\}!)^8\cdot(\{1\}!)^2\Bigr)}\notag\\
=&\frac{1}{s_1^3}\cdot\phi_5\Bigl(\frac{1}{\{2\}}\cdot
\bim{4}{2}\cdot\bim{4}{2}\cdot\bim{4}{2}\cdot\bim{3}{2}\cdot\bim{3}{2}
\Bigr).
\label{E:listofbinoms}
\end{align}
There is no way to attach
$\frac{1}{\{2\}}$ to any of the binomials in (\ref{E:listofbinoms})
to create a Catalan analogue.
Thus the corresponding term $\frac{1}{\{10\}_5}\bim{18}{9}_5$
is not necessarily polynomial.
Also note that in this particular example, the exponent of $s_1$ in the numerator is smaller than the exponent of $s_1$ in the denominator.
\end{Example}

The previous example highlights what is required for $\CatAB$
to be polynomial.
First, using Theorem \ref{T:Big2Theorem}, check that the pre-image of $\CatAB$ 
is polynomial in $\mathbb{Z}[\s,\t]$.
Second, check that the exponent of $s_1$
in the numerator is at least as large as the exponent of $s_1$
in the denominator. If these two conditions hold, $\CatAB$
is polynomial.

\begin{theorem}
If $m \pmod{r} < \frac{r}{2}$, then
\[\frac{1}{\{m+1\}\subn}\bim{2m}{m}\subn \]
is polynomial.
\end{theorem}

\begin{proof}
Since $m \pmod{r} < \frac{r}{2}$, $m$ can be expressed as $m = dr+a$, where $0\leq a < \frac{r}{2}$.  It follows that $2m = 2dr +2a$, where $0 \leq 2a < r$.  By Theorems  \ref{T:Matching} and \ref{T:Big2Theorem} and recalling that
\[\bim{2d+1}{d+1}\subn=\bim{2d+1}{d}\subn,
\]
yields
\[\bim{2m}{m}\subn=
\bim{2dr+2a}{dr+a}\subn = s_1^{\gamma_r(2dr+2a, dr+a)}\cdot\phi_r\Biggl(\bim{2d+1}{d+1}^{2a}\cdot \bim{2d}{d}^{r-2a}\Biggr).\]

\noindent
From (\ref{E:Gamma}), $\gamma_r(2dr+2a, dr+a) = a^2$.  By Lemma \ref{T:preimage}, we have that 
\[\{m+1\}\subn=\{dr+a+1\}\subn =s_1^{a}\cdot \phi_r(\{d+1\}).\]
Therefore,
\[\frac{1}{\{dr+a+1\}\subn}\bim{2dr+2a}{dr+a}\subn  =s_1^{a^2-a}\cdot \phi_r\Biggl(\bim{2d+1}{d+1}^{2a}\cdot\frac{1}{\{d+1\}}\bim{2d}{d}^{r-2a}\Biggr).\]
Note that $a^2-a \geq 0$ and $r-2a>0$.  By Theorem 4.1 in \cite{Sagan2018}, $\frac{1}{\{d+1\}}\bim{2d}{d}$ is polynomial and thus, by Lemma \ref{T:IsPoly}, so is $\Cat_{r}(m)=\Cat_{r}(dr+a)$.
\end{proof}



%

In general, 
for large values of $m$,
if $\Cat_r(m)$
is polynomial, then so is $\Cat_r(m+r)$.
This is due to the fact that $(m+r)\pmod{r}=m \pmod{r}.$
For small values of $m$, there may be some additional cancellations that do not happen generally.

\section{Multivariable Lucanomials}

\label{S:ToInfinity}

In this section, we extend the generalization of Lucas polynomials to an arbitrarily large set of variables.
Recall from (\ref{E:varepsilon}) that $\varepsilon_r(m)= ((m-1) \pmod{r}) +1$. 
Let $\mathcal{R} =(r_1,r_2,\ldots,r_\ello)$ be a decreasing sequence of positive integers and let $\varnothing$ denote the empty sequence.
Define the \emph{$\mathcal{R}$-Lucas polynomial} $\{m\}\suboneell$ recursively by
$\{m\}_{\varnothing} = s_1^{m-1}$ and 
\begin{equation}
\{m\}_{\mathcal{R}}=\{\varepsilon_{r_1}(m)\}_{(r_2,\ldots,r_\ell)}\cdot
\phi_{r_1}\Bigl(\Bigl\{\Bigl\lceil\frac{m}{r_1}\Bigr\rceil\Bigl\}\Bigr).
\label{D:MsubR}
\end{equation}
Notice that when $\mathcal{R}=(r_1)$, then $\{m\}_{\mathcal{R}} = \{m\}_{r_1}$.  As before, define $\{m\}_{\mathcal{R}}! = \{m\}_\mathcal{R}\cdot\{m-1\}_\mathcal{R}!$ where $\{0\}_\mathcal{R} ! = 1$.  We will use the notation $\mathcal{R}'$ for the sequence $(r_2,\ldots,r_\ello)$.

Recursively, define $\Delta_{m, \mathcal{R}}$ to be the collection of tiling words of $m$ that satisfy the following:
if $\mathcal{R} = \varnothing$, then 
$\Delta_{m, \varnothing}= \{\tau_1^{m}\}$,  otherwise, use the tiles from $\{\tau_{r_1},\tau_{2r_{1}}\}$ to tile the final $m-\varepsilon_{r_1}(m)$ tiles indiscriminately and tile the first $\varepsilon_{r_1}(m)$ tiles with a tiling word from $\Delta_{\varepsilon_{r_1}(m), \mathcal{R}'}$.

\begin{lemma}\label{L:CombInterpMsubR}
For $m \geq 0$, we have
\begin{equation}
\{m+1\}_\mathcal{R}=\sum_{\rmT\in \Delta_{m,\mathcal{R}}}\omega(\rmT),
\end{equation}
\end{lemma}

This gives a combinatorial method for computing $\{m\}_\mathcal{R}$.
\begin{Example}
Let $\mathcal{R} = (6, 2)$ and $m=18$.  The tilings corresponding to $\Delta_{17, (6, 2)}$

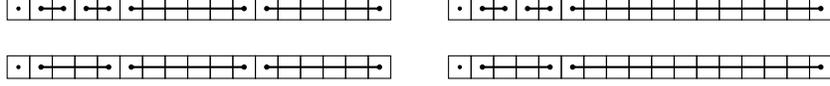
\begin{tikzpicture}[scale=0.3, every node/.style={scale=0.5}]
\draw[-](0,0) rectangle  (1,1); 
\draw[-](1,0) rectangle  (2,1); 
\draw[-](2,0) rectangle  (3,1); 
\draw[-](3,0) rectangle  (4,1); 
\draw[-](4,0) rectangle  (5,1); 
\draw[-](5,0) rectangle  (6,1); 
\draw[-](6,0) rectangle  (7,1); 
\draw[-](7,0) rectangle  (8,1); 
\draw[-](8,0) rectangle  (9,1); 
\draw[-](9,0) rectangle  (10,1); 
\draw[-](10,0) rectangle  (11,1); 
\draw[-](11,0) rectangle  (12,1); 
\draw[-](12,0) rectangle  (13,1); 
\draw[-](13,0) rectangle  (14,1); 
\draw[-](14,0) rectangle  (15,1); 
\draw[-](15,0) rectangle  (16,1); 
\draw[-](16,0) rectangle  (17,1); 
\mono{0}{0}

\hdom{1}{0}
\hdom{3}{0}

\hhex{5}{0}
\hhex{11}{0}
\end{tikzpicture}
\hspace{.5 cm}
\begin{tikzpicture}[scale=0.3, every node/.style={scale=0.5}]
\draw[-](0,0) rectangle  (1,1); 
\draw[-](1,0) rectangle  (2,1); 
\draw[-](2,0) rectangle  (3,1); 
\draw[-](3,0) rectangle  (4,1); 
\draw[-](4,0) rectangle  (5,1); 
\draw[-](5,0) rectangle  (6,1); 
\draw[-](6,0) rectangle  (7,1); 
\draw[-](7,0) rectangle  (8,1); 
\draw[-](8,0) rectangle  (9,1); 
\draw[-](9,0) rectangle  (10,1); 
\draw[-](10,0) rectangle  (11,1); 
\draw[-](11,0) rectangle  (12,1); 
\draw[-](12,0) rectangle  (13,1); 
\draw[-](13,0) rectangle  (14,1); 
\draw[-](14,0) rectangle  (15,1); 
\draw[-](15,0) rectangle  (16,1); 
\draw[-](16,0) rectangle  (17,1); 
\mono{0}{0}

\hdom{1}{0}
\hdom{3}{0}

\htwelve{5}{0}
\end{tikzpicture}

\bigskip
\begin{tikzpicture}[scale=0.3, every node/.style={scale=0.5}]
\draw[-](0,0) rectangle  (1,1); 
\draw[-](1,0) rectangle  (2,1); 
\draw[-](2,0) rectangle  (3,1); 
\draw[-](3,0) rectangle  (4,1); 
\draw[-](4,0) rectangle  (5,1); 
\draw[-](5,0) rectangle  (6,1); 
\draw[-](6,0) rectangle  (7,1); 
\draw[-](7,0) rectangle  (8,1); 
\draw[-](8,0) rectangle  (9,1); 
\draw[-](9,0) rectangle  (10,1); 
\draw[-](10,0) rectangle  (11,1); 
\draw[-](11,0) rectangle  (12,1); 
\draw[-](12,0) rectangle  (13,1); 
\draw[-](13,0) rectangle  (14,1); 
\draw[-](14,0) rectangle  (15,1); 
\draw[-](15,0) rectangle  (16,1); 
\draw[-](16,0) rectangle  (17,1); 
\mono{0}{0}

\hquad{1}{0}

\hhex{5}{0}
\hhex{11}{0}
\end{tikzpicture}
\hspace{.5 cm}
\begin{tikzpicture}[scale=0.3, every node/.style={scale=0.5}]
\draw[-](0,0) rectangle  (1,1); 
\draw[-](1,0) rectangle  (2,1); 
\draw[-](2,0) rectangle  (3,1); 
\draw[-](3,0) rectangle  (4,1); 
\draw[-](4,0) rectangle  (5,1); 
\draw[-](5,0) rectangle  (6,1); 
\draw[-](6,0) rectangle  (7,1); 
\draw[-](7,0) rectangle  (8,1); 
\draw[-](8,0) rectangle  (9,1); 
\draw[-](9,0) rectangle  (10,1); 
\draw[-](10,0) rectangle  (11,1); 
\draw[-](11,0) rectangle  (12,1); 
\draw[-](12,0) rectangle  (13,1); 
\draw[-](13,0) rectangle  (14,1); 
\draw[-](14,0) rectangle  (15,1); 
\draw[-](15,0) rectangle  (16,1); 
\draw[-](16,0) rectangle  (17,1); 
\mono{0}{0}

\hquad{1}{0}

\htwelve{5}{0}
\end{tikzpicture}
\\
\noindent
Thus, by Lemma \ref{L:CombInterpMsubR},
\[\{18\}_{(6,2)} =s_1s_2^2s_6^2 + s_1s_2^2s_{12} + s_1s_4s_6^2 + s_1s_4s_{12}   .\]
\end{Example}

Recursively, the generating function for $\{m\}_\mathcal{R}$ is given by
\begin{equation}
    L_\mathcal{R}(x)=\frac{\hat L_{\mathcal{R}',r_1}(x)}
    {1-s_{r_1}x^{r_1}-s_{2r_1}x^{2r_1}}
    \label{E:GenFuncRec}
\end{equation}
where
$\hat L_{\mathcal{R}',r_1}(x)$
is the polynomial that corresponds to the partial sum 
of the terms of degree $r_1$ and smaller in
 $L_{\mathcal{R}'}(x)$ and $L_{r}(x)$ is
given in (\ref{E:Recursion}).
The differences and similarities of this generating function
with that given
in (\ref{D:LucasGenFunct})
are worth noting.

\begin{Example}
The generating function for $\{m\}_{(5,2)}$ is given by
\begin{equation*}
L_{(5,2)}(x)=
\frac{x+s_1x^2+s_2x^3+s_1s_2x^4+(s_2^2+s_4)x^5
}
{1-s_5 x^{5}-s_{10}x^{10}}
\end{equation*}
since
$$L_2(x)=x+s_1x^2+s_2x^3+s_1s_2x^4+(s_2^2+s_4)x^5+\cdots.$$
\end{Example}

Properties of the $\mathcal{R}$-Lucas polynomials allow for a simple formula for the expansion of $\{m\}_{\mathcal{R}}!$.
In particular, we have the following.

\begin{lemma}
For $m \geq 1$ and $\mathcal{R} \neq \varnothing$, we have
\begin{equation}
\{m\}_\mathcal{R}!
=\Bigl(\{r_1\}_{\mathcal{R}'}!\Bigr)^{\bigl\lfloor\frac{m}{r_1}\bigr\rfloor}
\cdot
\{\varepsilon_{r_1}(m)\}_{\mathcal{R}'}!
\cdot
\phi_{r_1}
\Bigl(\prod_{j=0}^{r_1-1}
\Bigl\{\Bigl\lceil\frac{m-j}{r_1}\Bigr\rceil\Bigr\}!\Bigr).
\label{E:multfacexpan}
\end{equation}
\label{L:msubRfact}
\end{lemma}

\begin{proof}
The interval $[1, m]$ can be partitioned into $\bigl\lfloor\frac{m}{r_1}\bigr\rfloor$ intervals of the form $[gr_1+1,(g+1)r_1]$ and one interval of the form $[\lfloor\frac{m}{r_1}\rfloor\cdot r_1+1, m]$.
From (\ref{D:MsubR}), each of the $\bigl\lfloor\frac{m}{r_1}\bigr\rfloor$ intervals of the form $[gr_1+1,(g+1)r_1]$ 
contributes a factor of $\{r_1\}_{\mathcal{R}'}!$.  
The remaining interval
contributes the factor of
$\{\varepsilon_{r_1}(m)\}_{\mathcal{R}'}!$.  
The term
$$
\phi_{r_1}\Biggl(\Bigl\{\Bigl\lceil\frac{m-j}{r_1}\Bigr\rceil\Bigr\}\Biggr)
$$
is the analogue from Lemma \ref{T:applyphi}.
\end{proof}

We will now consider the binomial analogues of the $\mathcal{R}$-Lucas polynomials which is defined by 

\begin{equation}
 \bim{m}{n}\suboneell=
 \frac{\{m\}\suboneell!}{\{n\}\suboneell!\cdot\{m-n\}\suboneell!}.  
 \label{E:MchooseNsubR}
\end{equation}

\begin{Example}
Suppose $(r_1,r_2)=(9,3)$,
$m=52$, $n=31$, and thus $m-n=21.$
Then by Lemma \ref{L:msubRfact}, we have
\begin{align}
\{52\}_{(9,3)}!&=\Bigl(\{9\}_3!\Bigr)^5\cdot\{7\}_3!\cdot
\phi_9\Biggl(\Bigl(\{6\}!\Bigr)^7\cdot\Bigl(\{5\}!\Bigr)^2\Biggr),
\notag\\
\{31\}_{(9,3)}!&=\Bigl(\{9\}_3!\Bigr)^3\cdot\{4\}_3!\cdot
\phi_9\Biggl(\Bigl(\{4\}!\Bigr)^4\cdot\Bigl(\{3\}!\Bigr)^5\Biggr), \text{and}
\notag\\
\{21\}_{(9,3)}!&=\Bigl(\{9\}_3!\Bigr)^2\cdot\{3\}_3!\cdot
\phi_9\Biggl(\Bigl(\{3\}!\Bigr)^3\cdot\Bigl(\{2\}!\Bigr)^6\Biggr).
\notag
\end{align}
With the matching from Theorems \ref{T:Matching} and \ref{T:Big2Theorem}, we have
\begin{align}
\bim{52}{31}_{(9,3)}=&\bim{7}{4}_3\cdot
\phi_9\Biggl(
\bim{6}{3}^3
\cdot\bim{6}{4}^4
\cdot\bim{5}{3}^2
\Biggr)\notag\\
=
&s_1^0\cdot\phi_3\Biggl(\bim{3}{2}
\cdot\bim{2}{1}^2\Biggr)
\cdot\phi_9\Biggl(
\bim{6}{3}^3
\cdot\bim{6}{4}^4\notag
\cdot\bim{5}{3}^2\Biggr).
\end{align}
Note that $\bim{52}{31}_{(9,3)}$ is a polynomial in the five variables
$\{s_1,s_{3},s_{6},s_{9},s_{18}\}.$
\end{Example}

For $1\leq i \leq \ello$, let $\nu_1=m$,
$\nu_{i+1}=\varepsilon_{r_i}(\nu_{i})$,
$\alpha_1=n$, $\alpha_{i+1}=\varepsilon_{r_i}(\alpha_{i}),$ 
 $\beta_1=m-n$ and $\beta_{i+1}=\varepsilon_{r_i}(\beta_{i})$.
As can be seen in the previous example, 
the exponent of $\{9\}_3!$ in the numerator is greater than or equal to the exponent of $\{9\}_3!$ in the denominator.  
If
$\nu_{i}=\alpha_{i}+\beta_{i}$ 
and $\nu_i,\alpha_i,\beta_i\ge 0$, then
\begin{equation}
\Bigl\lfloor\frac{\nu_i}{r_i}\Bigr\rfloor
=\Bigl\lfloor\frac{\alpha_i+\beta_i}{r_i}\Bigr\rfloor
\ge
\Bigl\lfloor\frac{\alpha_i}{r_i}\Bigr\rfloor
+
\Bigl\lfloor\frac{\beta_i}{r_i}\Bigr\rfloor.
\label{e: floor}
\end{equation}
This ensures the cancellation like that of the previous example.

We will now give sufficient conditions for when the binomial analogues of the $\mathcal{R}$-Lucas polynomials is polynomial.

\begin{theorem}
If
$\nu_i=\alpha_i+\beta_i$ for $1\le i\le \ello+1$,
then 
\begin{equation*}
\bim{m}{n}\suboneell=
 \frac{\{m\}\suboneell!}{\{n\}\suboneell!\cdot\{m-n\}\suboneell!}  
\end{equation*}
is polynomial.
\label{T:multivariable}
\end{theorem}



\begin{proof}

The combination of Theorem \ref{T:Matching} and Theorem \ref{T:Big2Theorem} along with (\ref{E:multfacexpan}) and (\ref{e: floor}) immediately yields the result.
\end{proof}


%

The condition  that $\nu_i=\alpha_i+\beta_i$ is sufficient but not necessary for the corresponding polynomiality.

\begin{Example}
One can see that $\bim{76}{50}_{(15,5)}$ is polynomial even though
$\nu_2=1$, $\alpha_2=5$ and $\beta_2=11.$
In this particular case, everything cancels normally except for the  term
$$\frac{\{1\}_5\cdot\{15\}_5!}{\{5\}_5!\cdot\{11\}_5!}$$
which does in fact turn out to be polynomial. 
This corresponds to a special case of the form  $\alpha_2+\beta_2=r_{1}+1$ in which $\alpha_2\le r_2$.
\end{Example}

\section{Further Comments, Applications and Conclusions}


 



It is not hard to match up the equivalence classes $\hskip -6pt\pmod{r}$
of the integers so that the product of Lucanomials
in Theorem \ref{T:Matching}
get mapped to different sets of variables with different rules for tilings, left-right paths, or substitutions.
This would allow, for example, the multiplication of Fibbonnaci binomials with the $q$-binomials.
For such an application, however, one would need to be careful to match up the integers so that the corresponding substitutions are consistent with the matching of $M^*$, $N^*$ and $K^*$ given in Theorem \ref{T:Matching}.

\begin{Example}

With $S_i=\{s_{i,1},s_{i,2},\ldots\}$,
define $\Upsilon_{i}:\mathbb{Z}[s_r,s_{2r}]\rightarrow\mathbb{Z}[s_{i,r},s_{i,2r}]$ for $1\le i \le r$ by $\Upsilon_i(s_j)=s_{i,j}.$
Then apply $\Upsilon_i$ to the $i^{th}$ term in the 
binomial expansion. 
For example, (\ref{E:binexp}) would become
\begin{align}
&s_1\cdot\Upsilon_{1}\circ\phi_3\Biggl(\left\{
{\arraycolsep=1.4pt
\begin{array}{c}4\\2\end{array}}
\right\}\Biggr)
\cdot \Upsilon_{2}\circ\phi_3\Biggl(\left\{
{\arraycolsep=1.4pt
\begin{array}{c}3\\2\end{array}}\right\}\Biggr)
\cdot \Upsilon_{3}\circ\phi_3\Biggl(\left\{
{\arraycolsep=1.4pt
\begin{array}{c}3\\2\end{array}}
\right\}\Biggr)\notag\\
=&s_1\cdot\Upsilon_1\circ\phi_3\Bigl(\s^4+3\s^2\t+2\t^2\Bigr)
\cdot \Upsilon_2\circ\phi_3\Bigl(\s^2+\t\Bigr)
\cdot \Upsilon_3\circ\phi_3\Bigl(\s^2+\t\Bigr).\notag\notag
    \end{align}
    \end{Example}

The recursion formula for generating functions of $\{m\}_\mathcal{R}$ given in (\ref{E:GenFuncRec}) allows for a wide assortment of rational functions of the form
$\frac{Q(x)}{P(x)}$.  Classifying these rational functions into strict rules as to which Catalan analogues are polynomial would be a potential next step.


In this work, we present methods for defining multivariable Lucas polynomials.
It is worth pointing out that at no point in our work did we ever use properties
of Lucas polynomials other than
the the fact that their binomial and Catalan analogues are polynomial.  
Thus, this work has applications to other
collections of polynomials associated with integers whose binomial and Catalan analogues are polynomial.

\bibliography{references.bib} 

\begin{thebibliography}{BCMS18}

\bibitem[BCMS18]{Sagan2018}
Curtis Bennett, Juan Carrillo, John Machacek, and Bruce~E. Sagan.
\newblock Combinatorial interpretations of lucas analogues of binomial
  coefficients and catalan numbers.
\newblock {\em Preprint}, 2018.

\bibitem[GV85]{Gessel85}
Ira Gessel and Gerard Viennot.
\newblock Binomial determinants, paths, and hook length formulae.
\newblock {\em Adv. in Math}, 58:300--321, 1985.

\bibitem[Luc78a]{Lucas1878a}
E.~Lucas.
\newblock Theorie des fonctions numeriques simplement periodiques.
\newblock {\em Amer.J.Math.}, 1:184--196, 1878.

\bibitem[Luc78b]{Lucas1878b}
E.~Lucas.
\newblock Theorie des fonctions numeriques simplement periodiques [continued].
\newblock {\em Amer.J.Math.}, 1:197--240, 1878.

\bibitem[Luc78c]{Lucas1878c}
E.~Lucas.
\newblock Theorie des fonctions numeriques simplement periodiques [continued].
\newblock {\em Amer.J.Math.}, 1:289--321, 1878.

\bibitem[OEI]{OEIS3}
{\em OEIS Foundation Inc. (2020), The On-Line Encyclopedia of Integer
  Sequences, http://oeis.org/A000129}.

\bibitem[SS10]{Sagan2010}
Bruce~E. Sagan and Carla~D. Savage.
\newblock Combinatorial interpretations of binomial coefficient analogues
  related to lucas sequences.
\newblock {\em Integers}, 10:697--703, 2010.

\bibitem[ST19]{Sagan2019}
Bruce~E. Sagan and Jordan Tirrell.
\newblock Lucas atoms.
\newblock {\em Preprint}, 2019.

\end{thebibliography}

\end{document}